\documentclass[review,12pt,a4paper]{elsarticle}

\usepackage[algo2e,ruled,noline,linesnumbered]{algorithm2e}
\usepackage{algpseudocode}

\usepackage[utf8]{inputenc}
\usepackage{amsfonts,amsmath,amssymb,amsthm,mathrsfs}
\usepackage{graphicx}
\usepackage{color}
\usepackage{lineno,hyperref,tikz} \modulolinenumbers[5]

\def\R{\mathbb R}
\def\Pn{\mathcal P_n}
\def\Sn{\mathcal S_n}
\def\Mb{\mathcal M_n^\beta}
\def\wt{\widetilde}
\def\wh{\widehat}
\def\dis{\operatorname{d}}

\renewcommand{\theta}{\vartheta}

\DeclareMathOperator{\trace}{tr}
\DeclareMathOperator{\sign}{sign}
\DeclareMathOperator{\diag}{diag}
\DeclareMathOperator{\arctantwo}{arctan2}

\def\comment#1{}

\newtheorem{theorem}{Theorem}
\newtheorem{Lemma}[theorem]{Lemma}
\newtheorem{Corollary}[theorem]{Corollary}
\newtheorem{definition}[theorem]{Definition} 

\newtheorem{lemma}[theorem]{Lemma}
\newtheorem{remark}[theorem]{Remark} 
\newtheorem{example}[theorem]{Example}

\journal{Linear Algebra and its Applications}

\begin{document}

\title{Geometries on the cone of positive-definite matrices derived from the power potential and their relation to the power means} 

\author[enitaddress]{Nadia Chouaieb}
\ead{nadia.chouaieb@enit.utm.tn}
\author[dmiaddress]{Bruno Iannazzo}
\ead{bruno.iannazzo@dmi.unipg.it }
\author[enitaddress]{Maher Moakher\corref{mycorrespondingauthor}}
\ead[url]{http://sites.google.com/site/moakhersite/}
\ead{maher.moakher@enit.utm.tn}
\cortext[mycorrespondingauthor]{Corresponding author}

\address[enitaddress]{LAMSIN, National Engineering School of Tunis, University of Tunis El Manar, B.P. 37, 1002 Tunis-Belv\'ed\`ere, Tunisia}
\address[dmiaddress]{Dipartimento di Matematica e Informatica, Universit\`a di Perugia, Via Vanvitelli 1, I-06123 Perugia, Italy}

\begin{abstract}
We study a Riemannian metric on the cone of symmetric positive-definite matrices obtained from the Hessian of the power potential function $(1-\det(X)^\beta)/\beta$. We give explicit expressions for the geodesics and distance function, under suitable conditions. 
In the scalar case, the geodesic between two positive numbers coincides with a weighted power mean, while for matrices of size at least two it yields a notion of weighted power mean different from the ones given in the literature.
As $\beta$ tends to zero, the power potential converges to the logarithmic potential, that yields a well-known metric associated with the matrix geometric mean; we show that the geodesic and the distance associated with the power potential converge to the weighted matrix geometric mean and the distance associated with the logarithmic potential, respectively.
\end{abstract}


\maketitle

\section{Introduction} \label{sec:intro}

The importance of the cone of symmetric positive-definite matrices can hardly be exaggerated. Such matrices are omnipresent and play fundamental roles in several disciplines such as abstract mathematics, numerical analysis, probability and statistics and engineering sciences. Nowadays, as many applications deliver data that are constrained to live on this set, it has become essential to understand its geometric structure.

In recent years, a rather large number of metrics have been given on the cone of positive-definite matrices. There are both theoretical and practical reasons for this interest. Theoretically, it has been interesting to understand how to suitable extend notions such as the geometric mean or the power mean from positive scalars to positive-definite matrices
\cite{ll,alm,bhatia,limpalfia}. This generalization has required a profound understanding of the geometry on the cone of positive-definite matrices \cite[Ch. XII]{lang}, \cite[Ch. 6]{bhatia07}. Practically, having a large number of metrics, allows one to choose the one that best fits the problem of averaging or measuring the nearness of data provided by applications. Metrics different from the Euclidean one have been used in several applications that range from medical imaging, to machine learning and engineering, see e.g. \cite{ijp,cheng,lenglet05,hua,drasta,tuzel,barachant12,moakher05,moakher06,fi}.

On the set of symmetric positive-definite matrices of size $n$, that we will denote by $\Pn$, one can use the simple geometry inherited from the Euclidean geometry of symmetric matrices of size $n$, that we will denote by $\Sn$. Indeed, it is known that $\Pn$ is an open convex cone of the space of $\Sn$ and that the tangent space at a matrix $X \in \Pn$ can be identified with $\Sn$. The inner product on $\Sn$
\[
\langle A, B \rangle = \trace(A B),
\]
yields a Euclidean space structure on $\Sn$, which induces on $\Pn$ a natural structure of Riemannian manifold, as an open subset of a Euclidean space. However, the associated metric does not make $\Pn$ a complete metric space.

On the tangent space at $X\in\Pn$, one can also use the affine-invariant metric
\begin{equation} \label{eq:trmetric}
\langle A, B \rangle_X = \trace(X^{-1} A X^{-1} B),
\end{equation}
where $A,B$ are matrices in the tangent space, i.e., symmetric matrices. Endowing $\Pn$ with the metric~\eqref{eq:trmetric} induces on $\Pn$ a structure of Riemannian geometry that we will denote by $\mathcal{M}_n^0$ and $\Pn$ equipped with this metric is a complete metric space \cite{bhatia}.

It has been proved that in $\mathcal{M}_n^0$ any set $\{A_1, \ldots, A_m\}$ of given positive-definite matrices of size $n$ has a unique barycenter, that is a minimum over $\Pn$ of the function
\[
f(X) := \sum_{i=1}^m \delta^2(X, A_i),
\]
where $\delta(A,B)$ is the Riemannian distance between two given matrices $A$ and $B$ in $\mathcal{M}_n^0$. The barycenter on $\mathcal{M}_n^0$ has been understood as the legitimate generalization of the geometric mean to matrices \cite{moakher05b}, and it is often referred to as the Karcher mean.

Moreover, there exists a unique geodesic $\gamma : [0,1] \to \Pn$ joining any two matrices $A, B \in \mathcal{M}_n^0$, and a point of the geodetic curve $\gamma(t)$ is said to be the weighted matrix geometric mean of $A$ and $B$ \cite{lll}.
 
The Riemannian metric~\eqref{eq:trmetric} of $\mathcal{M}_n^0$ can also be viewed as the Hessian of the logarithmic potential function on $\Pn$
\begin{equation} \label{eq:0pot}
\Phi_0(X) = - \ln(\det X).
\end{equation}
In interior-point methods of cone programming, the function~\eqref{eq:0pot} is called the logarithmic barrier function \cite[Sec. 6.3]{nt}. Note that $\ln(\det X) = \trace \ln X$, and hence the negative of the function~\eqref{eq:0pot} is called in classical statistical mechanics the Boltzmann entropy and in information theory it is called the information potential \cite{Hiai2014}.

In this paper, we consider the $\beta$-power potential function on $\Pn$ \cite{ohara,ohara13}
\begin{equation} \label{eq:betapot}
\Phi_\beta(X) = \frac{1 - (\det X)^\beta}{\beta}, \quad \beta \ne 0,
\end{equation}
that generalizes the logarithmic potential in the sense that $\lim_{\beta\to 0} \Phi_\beta(X) = \Phi_0(X)$. For $0 \ne \beta < \frac1n$, the Hessian of $\Phi_\beta(X)$ provides a new family of Riemannian geometries on $\Pn$ that generalizes $\mathcal{M}_n^0$ and will be denoted by $\Mb$.

We will show that the $\beta$-potential function~\eqref{eq:betapot} is intimately related to generalized logarithmic functions. In fact, in his theory of non-extensive statistics mechanics, which is a generalization of the Boltzmann-Gibbs statistical mechanics, Tsallis introduced in 1994 \cite{tsallis} the one-parameter family, parameterized by $q \in \mathbb{R}$, of generalized logarithmic functions 
\[
\ln_q x = \left\{\begin{array}{ll} \ln x,& q=1,\\
\frac{x^{1-q}-1}{1-q},& q\ne 1,\end{array}\right.
\]
which are defined for all $x>0$. The function, $\ln_q$, is called the $q$-logarithm (or the Tsallis logarithm). We note that the family of $q$-logarithms can be seen as the one-parameter family of Box-Cox transformations (with deformation parameter $\lambda = 1 - q$) which was proposed by Box and Cox in 1964 \cite{box-cox}.

By setting $\hat\beta = 1-\beta$, the $\beta$-power potential function~\eqref{eq:betapot} can be compactly written as
\[
\Phi_\beta(X) = - \ln_{\hat\beta}(\det X), \quad  X \in \Pn,
\]
and therefore, the $\beta$-power potential function can also be called the $\hat\beta$-logarithmic potential function on $\Pn$. 

We study the geometry of $\Mb$ and derive an explicit expression for the geodesics and for the Riemannian distance on $\Mb$, under some conditions. In the scalar case, we observe that the geodesics joining the two positive numbers $a$ and $b$ is the weighted power mean $((1-t) a^{\beta/2} + t b^{\beta/2})^{2/\beta}$. This is nice, since it is known that the weighted power mean with parameter $\beta/2$, for a given $t$ converges to the weighted geometric mean $a^{1-t} b^t$ as $\beta \to 0$. For $n>1$, this symmetry disappears and the geodesic on $\Mb$, even for two matrices, does not coincide, for any choice of the parameter, with the definitions of weighted power mean provided so far, namely the straightforward one $((1-t) A^q + t B^q)^{1/q}$, for $t\in[0,1]$, and $q \in \R \setminus \{0\}$, and the one provided by Lim and P\'alfia~\cite{limpalfia}. Thus, the barycenter in $\Mb$ can be seen as a third definition of power mean. The latter has the luxury of being derived from a rich geometric structure on $\Pn$ and it approaches the weighted geometric mean $A \#_t B = A(A^{-1} B)^t$ when the parameter $\beta$ goes to 0.

We here recall that for a diagonalizable matrix $C \in \R^{n\times n}$ with positive eigenvalues, that is, such that there exist an invertible matrix $M$ and a diagonal matrix $D = \diag(d_1, \ldots, d_n)$ such that $C = M^{-1} D M$, $C^q$ should be understood as the primary matrix function $C^q = M^{-1} \diag(d_1^q, \ldots, d_n^q)M$. A remarkable property of this function, that will be useful in the following is the commutativity with similarities, i.e., if $N$ is invertible, then $N C^q N^{-1} = (N C N^{-1})^q$.

The paper is structured as follows: in Section~\ref{sec:bpow} we introduce the metric related to the $\beta$-power potential, and study some of its properties; in Section~\ref{sec:geo} we provide explicit expressions for the geodesic joining two positive-definite matrices in the scalar, linearly dependent, and general cases; a number of properties of the mean of two positive-definite matrices associated with this metric as well as different asymptotic results are studied in Section~\ref{sec:mean}. Explicit expression of the Riemannian distance between two positive-definite matrices is derived in Section~\ref{sec:dist}, that, as $\beta$ goes to 0, is shown to converge to the Riemannian distance associated with the metric \eqref{eq:trmetric}. Finally, in Section~\ref{sec:conc} some conclusions are drawn.

\section{A geometry on $\Pn$ related to the $\beta$-power potential} \label{sec:bpow}

The logarithmic potential function $\Phi_0(X) = - \ln(\det(X))$ on $\Pn$ is strictly related to the matrix geometric mean. Indeed, the differential of $\Phi_0(X)$ is 
\[
d \Phi_0(X) = - \trace(X^{-1} dX),
\]
and the second differential of $\Phi_0(X)$ is 
\begin{equation}\label{eq:qf1}
d^2 \Phi_0(X) = \trace(X^{-1} dX X^{-1} dX),
\end{equation}
which is a positive-definite quadratic form, whose associated metric, is the one defined in~\eqref{eq:trmetric}. The barycenter with respect to~\eqref{eq:trmetric} is said to be the matrix geometric mean or the Karcher mean.

The one-parameter family of power potential functions~\eqref{eq:betapot}, which is a generalization of the logarithmic potential in the sense that $\lim_{\beta\to 0} \Phi_\beta(X)=\Phi_0(X)$, provides for all values of $\beta \in (- \infty, 0) \cup (0, \frac1n)$ a family of Riemannian metrics that are thus very natural to study.

Indeed, since 
\[
d \Phi_\beta(X) = - (\det X)^\beta \trace(X^{-1} dX),
\]
we have that the second differential of $\Phi_\beta(X)$ is 
\begin{equation} \label{eq:qf2}
d^2 \Phi_\beta(X) = (\det X)^\beta \left( \trace(X^{-1} dX X^{-1} dX) - \beta \trace^2(X^{-1} dX) \right),
\end{equation}
where we have used $d\det(X) = \det(X) \trace(X^{-1}dX)$ and $d(X^{-1}) = - X^{-1}dX X^{-1}$.

The quadratic form~\eqref{eq:qf2} is positive definite as long as $\beta < \frac{1}{n}$. To prove this fact, we can write~\eqref{eq:qf2} as
\begin{equation*}
d^2 \Phi_\beta(X) = (\det X)^\beta \bigl( \trace((X^{-1/2} dX X^{-1/2})^2) - \beta \trace^2(X^{-1/2} dX X^{-1/2}) \bigr),
\end{equation*}
and use the following.
\begin{lemma} \label{lem:posdef}
For any $n \times n$ symmetric matrix $S \ne 0$ and $\beta < \frac1n$ we have
\[
\trace(S^2)  - \beta \trace^2(S) > 0.
\]
\end{lemma}

\begin{proof}
This follows from the Cauchy-Schwartz inequality:
\[
\beta\trace^2 (S) = \beta\langle S, I \rangle^2 \le \beta\langle I, I \rangle \langle S, S \rangle = \beta n\trace (S^2) < \trace(S^2).
\]
\end{proof}

The Riemannian metric associated with \eqref{eq:qf2} at the base point $X \in \Pn$ is given by
\begin{equation} \label{eq:betametric}
g_{X}^{\beta}(A, B) :=  (\det X^\beta) \bigl( \trace(X^{-1} A X^{-1} B) - \beta \trace(X^{-1} A) \trace(X^{-1} B) \bigr), 
\end{equation}
where $A$ and $B$ are points of the tangent space to $\Pn$ at $X$, identified as usual with $\Sn$.

We will denote by $\Mb$ the Riemannian manifold obtained by endowing $\Pn$ with the metric~\eqref{eq:betametric}. In Section \ref{sec:geo}, an explicit expression for the geodesics of $\Mb$ will be provided, but first we describe a wide set of isometries of $\Mb$.

\begin{Lemma}\label{thm:lem1v}
Let $M \in GL(n,\R)$ be such that $\det(M) = \pm 1$. The function $f : \Pn \to \Pn$ such that $f(X) = M X M^T$ is an isometry of $\Pn$ endowed with the metric~\eqref{eq:betametric}.
\end{Lemma}

\begin{proof}
Since $f$ is linear, we have that $Df(X)[H] = M H M^T$ for any positive-definite matrix $X$ and symmetric matrix $H$. Hence,
\[
g_{f(X)}^\beta(Df(X)[A], Df(X)[B]) = g_{MXM^T}^\beta(MAM^T, MBM^T) = (\det(M)^2)^\beta g_X^\beta(A, B),
\]
for any $A,B\in\Sn$, and this completes the proof.
\end{proof}

For the reader's convenience and in order to keep the paper self-contained, we give here the definition of a totally geodesic submanifold. Then, we recall a well-known result in Riemannian geometry that characterizes totally-geodesic submanifolds.  
\begin{definition}[\cite{chen2000, helgason1978}] 
A submanifold $N$ of a Riemannian manifold $(M, g)$ is called totally geodesic if any geodesic on the submanifold $N$ with its induced Riemannian metric is also a geodesic on the Riemannian manifold $(M, g)$.  
\end{definition}

\begin{theorem}[Thm. 5.1 in \cite{kobayashi}]\label{thm:kobayashi}
Let $\mathcal{M}$ be a Riemannian manifold and $\mathfrak S$ any set of isometries of $\mathcal{M}$. Let $F$ be the set of points of $\mathcal{M}$ which are left fixed by any elements of $\mathfrak S$. Then each connected component of $F$ is a closed totally-geodesic submanifold of $\mathcal{M}$.
\end{theorem}

This allows us to identify two totally-geodesic submanifolds of $\Mb$ that are useful in the following.

\begin{Corollary}\label{thm:tot}
The set of positive-definite diagonal matrices of size $n$ and the set $\{\alpha I\,:\,\alpha>0\}$ of positive scalar matrices of size $n$  are totally-geodesic submanifolds of the Riemannian manifold $\Pn$ endowed with the metric~\eqref{eq:betametric}.
\end{Corollary}
\begin{proof}
Let $M_i\in\R^{n\times n}$ be the diagonal matrix such that $(M_i)_{ii}=-1$ and $(M_i)_{jj}=1$, for $j\ne i$. By Lemma \ref{thm:lem1v}, the map $X\to M_iXM_i^T$ is an isometry of $\Mb$, and its fixed points are the positive-definite matrices whose $i$-th row and column are $0$ except in the position $i$. 

The common fixed points of all the isometries of the set $\mathfrak S=\{M_1,\ldots,M_n\}$ are the positive diagonal matrices, and by Theorem \ref{thm:kobayashi} they form a totally-geodesic submanifold of $\Mb$.

Now consider, for $i=2,\ldots,n$, the matrix $N_i\in\R^{n\times n}$ that permutes the components $1$ and $i$ of a vector in $\R^n$. We have that $(N_i)_{h\ell}=1$ for (a) $h=\ell$, with $h\not\in\{1,i\}$; (b) $h=1$, $\ell=i$ and; (c) $h=i$, $\ell=1$; and $(N_i)_{h\ell}=0$ elsewhere. A fixed point $X$ of the isometry $X\to N_iXN_i^T$ is such that $X_{11}=X_{ii}$. 

The fixed points common to the isometries in the set $\mathfrak T=\{M_1,\ldots,M_n,N_2,\ldots,N_n\}$ are the diagonal matrices with constant diagonal, namely the positive scalar matrices, that by the aforementioned theorem form a totally-geodesic submanifold of $\Mb$.
\end{proof}

\section{Geodesics of $\Mb$} \label{sec:geo}

We show that there exists a unique geodesic curve for the metric~\eqref{eq:qf2} joining two positive-definite matrices $A$ and $B$, under some conditions on $\beta$. Moreover, we provide an explicit expression for the geodesic,  by solving analytically the differential equation that it satisfies.

In the scalar case, the geodesic curve can be interpreted as the weighted $p$-power mean, for $p=\beta/2$, of two positive numbers $a,b$, namely $((1-t)a^p+tb^p)^{1/p}$, where $t\in[0,1]$ is the weight and $p$ is a nonzero real number.
This is a very nice feature, since the geodesics for the metric~\eqref{eq:qf1}, corresponding to the logarithmic potential is the weighted matrix geometric mean $a^{1-t}b^t$, that, for a given $t$, is the limit for $p\to 0$ of the $p$-power mean.

We show that a similar result holds when $A$ and $B$ are two linearly dependent matrices, where the geodesics turn out to be $((1-t)A^p+tB^p)^{1/p}$, for $p=n\beta/2$. For linearly independent $A$ and $B$, the explicit expression is much more complicated.

\subsection{Geodesic equation}

Let $\mathscr{C}_{A,B}$ denote the space of all $C^2$-curves from the interval $[0, 1]$ to $\Pn$ that join two points $A$ and $B$ in $\Pn$, i.e.,
\[
\mathscr{C}_{A,B} := \left\{ P : [0, 1] \rightarrow \Pn, P \in C^2([0, 1]) \ | \ P(0) = A, \ P(1) = B \right\}.
\]
A geodesic joining $A$ and $B$ in $\Mb$ is a curve in $\mathscr{C}_{A,B}$ that minimizes the length with respect to the metric $\Mb$, that is, the functional 
\begin{equation}\label{eq:length}
\mathscr{L}^\beta(P) := \int_0^1 \sqrt{g_{P(t)}^\beta(P'(t),P'(t))} \, dt.
\end{equation}

In the next result we derive a differential equation whose solution provides the geodesics.

\begin{theorem}
Let $P : [0, 1] \rightarrow \Pn$ be a smooth geodesic on $\Pn$ equipped with the Riemannian metric~\eqref{eq:betametric}. Then the function $G(t) = P^{-1}(t) {P}'(t)$ satisfies the differential equation
\begin{equation} \label{eq:GeodEq}
G' = \frac{\beta}{2(1 - n \beta)} \left( \trace(G^2) - \beta \trace^2(G) \right) I - \beta \trace(G) G.
\end{equation}
\end{theorem}
\begin{proof} 
It is a known fact in differential geometry (see e.g., \cite[p.~17]{jost}) that the extremal curves for the length functional $\mathscr{L}^\beta$ coincide with the extremal curves for the energy functional defined by 
\[
\mathscr{E}^\beta(P) := \frac{1}{2}\int_0^1 g_{P(t)}^\beta(P'(t),P'(t)) \, dt. 
\]
A customary technique to find the critical curves of $\mathscr{E}^{\beta}$, is to solve the Euler-Lagrange equation associated with $\mathscr{E}^{\beta}$, that is
\[
\frac{d}{dt} \frac{\partial L}{\partial P'} - \frac{\partial L}{\partial P} = 0,
\]
where $L(P, P') := \tfrac{1}{2} g_{P(t)}^\beta(P'(t),P'(t))$ is the ``Lagrangian'' associated with the energy functional $\mathscr{E}^\beta$. 

Direct computation gives
\begin{align*}
\frac{\partial L}{\partial P} &= (\det P)^\beta \left[ \frac{\beta}2 \left(\trace(G^2) - \beta \trace^2(G) \right) I - G^2 + \beta \trace(G) G \right] P^{-1}, \\
\intertext{and}
\frac{\partial L}{\partial P'} &= (\det P)^\beta \left[ G - \beta \trace(G) I \right] P^{-1},
\end{align*}
where we have set $G = P^{-1} P'$.
Taking the derivative of the previous equation with respect to $t$ and after some work we obtain
\begin{equation}\label{eq:ode-G}
G'-\beta \trace(G')I = -\beta \trace(G)G+\frac{1}{2}\beta^2\trace^2(G)I + \frac{1}{2}\beta\trace(G^2)I.
\end{equation} 
By applying the trace to both sides of equation~\eqref{eq:ode-G}, we get
\begin{equation*} 
\trace(G') = \frac{\beta}{2(1 - n\beta)} \left( 
n\trace(G^2) + (\beta n-2) \trace^2(G) \right),
\end{equation*}
which, for later convenience, we rewrite as
\begin{equation} \label{eq:ode-trG} 
\trace(G') = \frac{n\beta}{2(1 - n\beta)} \left( 
\trace(G^2) - \beta \trace^2(G) \right) - \beta \trace^2(G).
\end{equation}
Substituting \eqref{eq:ode-trG} into \eqref{eq:ode-G}, after some simple manipulations, yields \eqref{eq:GeodEq}.
\end{proof}

To find geodesic curves one can solve equation \eqref{eq:GeodEq}. To this end, we found it useful to decompose $G(t)$ into its isotropic part $\alpha(t) I$, where $\alpha(t)=\frac{1}{n}\trace(G(t))$, and its deviatoric part $\wt{G}(t) = G(t) - \alpha(t) I$ (recall that this decomposition is unique for any given symmetric matrix, see e.g. \cite{augusti}). We then have $\trace G(t) = n \alpha(t)$ and  $\trace \wt{G}(t) = 0$ and  equation \eqref{eq:GeodEq} can be written as
\begin{equation} \label{eq:GeodEq2}
\left\{\begin{array}{l}
\alpha' = - n \beta \left( \frac{1}{2} \alpha^2 - \frac{1}{2n(1 - n \beta)} \trace(\wt{G}^2) \right), \\
\wt{G}' =  - n \beta \alpha \wt{G}.
\end{array}\right.
\end{equation}
We will get a closed form solution of equation~\eqref{eq:GeodEq2} when $G(t)$ is a diagonal matrix. In Section \ref{sec:diag} we will show that this is not a restriction.

\subsection{Reduction to diagonal matrices}\label{sec:diag}

Let $P:[0, 1]\to \Pn$ be a smooth curve, such that $P(0)=A$ and $P(1)=B$, for $A,B\in\Pn$. The length of $P$ in $\Mb$ is defined in~\eqref{eq:length}.

There exists $M\in GL(n)$ such that $M^TAM=I$ and $M^TBM=D$ is a positive definite diagonal matrix (one can choose $M=A^{-1/2}U$, where $U$ is an orthogonal matrix such that $D:=U^TA^{-1/2}BA^{-1/2}U$ is diagonal).

If we set $Q(t)=M^TP(t) M$, then  we have
\[
\begin{split}
& \trace(Q(t)^{-1}Q'(t)Q(t)^{-1}Q'(t))=\trace(P(t)^{-1}P'(t)P(t)^{-1}P'(t)),\\
&
\trace^2(Q(t)^{-1}Q'(t))=\trace^2(P(t)^{-1}P'(t)), 
\end{split}
\]
and $\det(Q(t))^\beta=|\det(M)|^{2\beta}\det(P(t))^\beta$, from which we obtain
\[
g_{Q(t)}^\beta(Q'(t),Q'(t))
=|\det(M)|^{2\beta}g_{Q(t)}^\beta(P'(t),P'(t)),
\]
and thus 
\[
\mathcal L(Q(t))=|\det(M)|^\beta\mathcal L(P(t)).
\]
The constant $|\det(M)|^\beta = \det(A)^{-\beta/2}$ does not depend on $P(t)$, but just on $A$, thus, for any smooth curve $P(t)$ joining $A$ and $B$ there is a smooth curve joining $I$ and $D$ whose length is a scalar multiple of the length of $P(t)$, and also the converse is true. 

This shows that, without loss of generality, in order to find the geodesics, we can consider only curves joining the identity and a diagonal matrix with positive diagonal entries.

\subsection{Special cases}\label{sec:special}

Before giving the general form of the geodesic, we consider the case in which the endpoints are both multiples of the same matrix or, equivalently, are linearly dependent. This includes, in particular, the case $n=1$. For these matrices, the geodesic has a simple form and it is related to the scalar weighted power mean.

\begin{lemma}\label{thm:lem7}
Let $A\in\Pn$. A ray $\{ rA, \ r > 0 \}$ is a totally-geodesic submanifold of $\Mb$ and the geodesic joining $r_0 A$ and $r_1 A$ is given by
\begin{equation}\label{eq:i}
P(t) = \bigl( (1 - t) r_0^{\frac{n\beta}2} + t r_1^{\frac{n\beta}2} \bigr)^{\frac2{n\beta}} A,\qquad t\in[0,1].
\end{equation}
\end{lemma}  
\begin{proof}
We consider first the case $A=I$. By Corollary \ref{thm:tot}, the set of scalar matrices is totally geodesic, and then the geodesic joining $r_0I$ and $r_1I$ should be of the type $\xi(t)I$, where $\xi(t)$ is a positive number. 

In the system~\eqref{eq:GeodEq2}, we can set $\wt{G} = 0$, and we need to solve the two-point boundary-value problem
\begin{equation}\label{eq:alpha1}
\left\{
\begin{array}{l}
\alpha' = - \frac{n \beta}2 \alpha^2,\\
\xi' \xi^{-1} = \alpha,\\
\xi(0)=r_0,\quad \xi(1)=r_1.
\end{array}
\right.
\end{equation}
The corresponding initial-value problem, for given $\alpha(t_0)$ and $\xi(t_0)>0$, has a unique local solution for $t_0\in\R$. If $\alpha(t_0)=0$, then the unique solution is $\alpha\equiv 0$ with $\xi(t)\equiv\xi(t_0)$.

We can solve the initial-value problem with $\alpha(0)=c\ne 0$ and $\xi(0)=r_0$, by separation of variables, obtaining $\xi(t)=\bigl(\frac{n\beta}{2}ct+1\bigr)^{2/(n\beta)}r_0$, defined in a right neighborhood of $0$. Setting $\xi(1)=r_1$, we get $c=\frac{2}{n\beta}\bigl((\frac{r_1}{r_0})^{n\beta/2}-1\bigr)$. With this choice of $c$, nonzero for $r_1\ne r_0$, the initial-value problem with $\alpha(0)=c$ and $\xi(0)=r_0$, has a unique solution in $[0,1]$, such that $\xi(1)=r_1$, namely
\[
\xi(t) = \bigl( (1 - t) r_0^{\frac{n\beta}2} + t \; r_1^{\frac{n\beta}2} \bigr)^{\frac2{n\beta}},
\] 
that in turn is the unique solution of the boundary-value problem. For $r_0=r_1$ the unique solution is $\alpha\equiv 0$ and $\xi(t)\equiv r_1$.

Using the argument of Section \ref{sec:diag}, with $M=A^{-1/2}$, we have that a ray is a totally-geodesic submanifold. In particular, we have $Mr_0AM^T=r_0I$ and $Mr_1AM^T=r_1I$, and thus the geodesic joining $r_0A$ and $r_1A$ is the curve $A^{1/2}\xi(t)A^{1/2} = \xi(t)A$, that is~\eqref{eq:i}.
\end{proof}

Lemma~\ref{thm:lem7} can be restated in the following way:
\begin{Corollary}\label{thm:lindep}
Let $A$ and $B$ be linearly dependent matrices in $\Pn$. Then, there exists a unique geodesic in $\Mb$ joining $A$ and $B$ given by
\begin{equation} \label{eq:lindep}
G_\beta(A, B, t) = \bigl((1-t)A^{\frac{n\beta}{2}} + t B^{\frac{n\beta}{2}}\bigr)^{\frac{2}{n\beta}}, \qquad t \in [0,1].
\end{equation}
\end{Corollary}

Thus, for linearly dependent matrices, the geodesics with respect to the metric in $\Mb$ has the form of the weighted power mean with parameter $n\beta/2$, obtained by the scalar weighted power mean by substituting scalars with matrices.

As an interesting corollary we get the explicit geodesic in the scalar case that coincides with a weighted power mean.

\begin{Corollary} \label{thm:scalar}
The geodesic joining $a, b \in \mathcal{P}_1$ endowed with the metric~\eqref{eq:betametric} is the weighted power mean with parameter $\beta/2$,
\begin{equation} \label{eq:geoscalar}
G_\beta(a, b, t) = \bigl( (1 - t) a^{\frac{\beta}2} + t b^{\frac{\beta}2} \bigr)^{\frac2{\beta}}, \qquad t \in[0,1].
\end{equation}
\end{Corollary}

\subsection{The general case}

We provide our main results, that is, a general form of the geodesic on $\Mb$ under suitable conditions. To simplify the exposition, we first need to set some notation and introduce some variables that simplify the expression of the geodesic. For $A,B\in\Pn$ we set $\wt A=\det(A)^{-1/n}A$, $\wt B=\det(B)^{-1/n}B$, that are the scaled versions with determinant one of $A$ and $B$, respectively. Let $\wt \mu_1,\ldots,\wt \mu_n$ be the eigenvalues of $\wt A^{-1}\wt B$, and define
\begin{equation} \label{eq:zeta}
\zeta_i = \ln \wt \mu_i, \qquad i=1, \ldots, n.
\end{equation}
Recalling that the Riemannian distance on $\mathcal{M}_0$ between two matrices $M, N \in \Pn$ is 
\begin{equation} \label{eq:rdist}
\delta(M, N) := \Bigl(\sum_{i=1}^n \ln^2 \lambda_i\Bigr)^{1/2},
\end{equation}
where $\lambda_1, \ldots, \lambda_n$ are the eigenvalues of $M^{-1}N$, we observe that the norm of the vector $\zeta = [\zeta_1, \ldots, \zeta_n]$ is nothing but the Riemannian distance on $\mathcal{M}_0$ between $\wt A$ and $\wt B$, i.e., 
\[
\|\zeta\| = \Bigl(\sum_{i=1}^n\ln^2 \wt \mu_i\Bigr)^{1/2} = \delta(\wt A, \wt B).
\]

We are now in a position to define, for $\beta\in(-\infty,0)\cup(0,1/n)$, the following measure of linear \textcolor{blue}{in}dependence that will play important role in subsequent developments
\begin{equation} \label{eq:gamma}
\gamma_\beta(A, B) := \frac{|\beta|\|\zeta\|}{2\sqrt{1/n-\beta}} = \frac{|\beta|\delta(\wt A,\wt B)}{2\sqrt{1/n-\beta}} = \frac{|\beta|\delta(\det(A)^{-1/n} A,\det(B)^{-1/n} B)}{2\sqrt{1/n-\beta}},
\end{equation}
with $\delta$ as in \eqref{eq:rdist}. In fact, $\gamma_\beta(\cdot, \cdot)$ is a distance function on the quotient space $\Pn/\sim$ where the equivalence relation is defined as $A \sim B$ if $A$ and $B$ are on the same ray, i.e., are multiple of one another. Some interesting properties of $\gamma_\beta(A,B)$ are summarized in the following.

\begin{Lemma}\label{thm:gamma}
For any $\beta\in(-\infty,0) \cup (0,1/n)$, the function $\gamma_\beta : \Pn \times \Pn \to \mathbb R$, defined by~\eqref{eq:gamma}, satisfies for all $A, B, C \in \Pn$
\begin{enumerate}
\item\label{item:gamma0} Symmetry: $\gamma_\beta(A,B)=\gamma_\beta(B,A)$;

\item Positive definiteness: $\gamma_\beta(A,B) \ge 0$ with equality if and only if $A$ and $B$ are linearly dependent;

\item Triangle inequality: $\gamma_\beta(A,C) \le \gamma_\beta(A,B) + \gamma_\beta(B,C)$;

\item Invariance under inversion: $\gamma_\beta(A^{-1},B^{-1}) = \gamma_\beta(A,B)$;

\item Invariance under congruence: if $M$ is invertible, then $\gamma_\beta(M^TAM,M^TBM)=\gamma_\beta(A,B)$.
\end{enumerate}
\end{Lemma}

\begin{proof}
These properties follow directly from the analogous properties of the Riemanian distance $\delta$, see e.g.~\cite{moakher05b}.
\end{proof}

For later use, we note that if $A$ and $B$ are linearly independent then $d := \delta(\wt A, \wt B)\ne 0$ and $\hat\beta_1$ and $\hat\beta_2$ given by
\begin{equation} \label{eq:beta12}
\hat\beta_1 := - 2\pi \frac{\sqrt{\pi^2n^2+nd^2}+\pi n}{nd^2}, \quad
\hat\beta_2 := 2\pi \frac{\sqrt{\pi^2n^2+nd^2} -\pi n}{nd^2},
\end{equation}
are well defined and satisfy $\hat\beta_1  < 0 < \hat\beta_2 < 1/n$. We have that $0<\gamma_\beta(A, B) <\pi$ if and only if $\beta\in(\hat\beta_1,0)\cup(0,\hat\beta_2)$. We also note that we have $\hat\beta_1\to -\infty$ and $\hat\beta_2\to 1/n$ as $d\to 0$.

Our main result is that, under some restrictions on $\beta$, we have the existence and the explicit expression of the geodesic joining two matrices in $\Pn$.
But before stating the main theorem, we give the following two results in the case of diagonal matrices. 

\begin{Lemma}\label{thm:lemma}
Let $D_A,D_B$ be two diagonal matrices in $\Pn$, $\beta\in(-\infty,0)\cup(0,1/n)$ and $\gamma:=\gamma_\beta(D_A,D_B)$ as in~\eqref{eq:gamma}. If $0<\gamma<\pi$, then there exists a unique geodesic $\wt P(t)$ on $\Mb$ joining $D_A$ and $D_B$ such that $\wt P(t)=\diag(\lambda_1(t),\ldots,\lambda_n(t))$, where\footnote{The two-argument (or four-quadrant) inverse tangent function, $\arctantwo(x, y)$, returns values in the closed interval $[-\pi, \pi]$ based on the values of $x$ and $y$, as opposed to $\arctan(y/x)$ which returns values in the closed interval $[-\pi/2, \pi/2]$, see e.g. 
\url{https://www.mathworks.com/help/matlab/ref/atan2.html}.}
\[
\begin{split}
	\lambda_i & = \biggl((1-t)(1-t+t\sigma\cos\gamma)\exp\Bigl(\frac{n\beta}{\gamma}\arctantwo(t\sigma\sin\gamma,1-t+t\sigma\cos\gamma)\zeta_i\Bigr)\\
                        & +t(t+(1-t)\sigma^{-1}\cos\gamma)\exp\Bigl(-\frac{n\beta}{\gamma}\arctantwo((1-t)\sin\gamma,(1-t)\cos\gamma +t\sigma)\zeta_i\Bigr)
				    \biggr)^{\frac{1}{n\beta}}
\end{split}
\]
for $i=1,\ldots,n$ and $t\in[0,1]$, with 
\[	
\sigma=\det(D_A^{-1}D_B)^{\beta/2},\quad \zeta_i=\ln\Bigl(\frac{\lambda_i(1)}{\lambda_i(0)}\sigma^{-2/(n\beta)}\Bigr), \quad \gamma=\frac{|\beta|\| \zeta \|}{2\sqrt{1/n-\beta}}.
\]
\end{Lemma}
For the convenience of the reader, the rather lengthy proof of this Lemma is given in the Appendix. 

The condition $\gamma<\pi$ may appear to be a bothering restriction, however it cannot be relaxed. Indeed, by observing that for a given couple $(A,B)$ of positive-definite matrices, there exist $\wt \beta_1<0<\wt \beta_2$ such that $\gamma(A,B;\wt \beta_1)=\gamma(A,B,\wt \beta_2)=\pi$ and $\gamma(A,B,\beta)<\pi$ for $\beta\in(\wt \beta_1,0)\cup(0,\beta_2)$, we get
\[
\lim_{\beta\to\beta_1^+} \lambda_i\Bigl(\frac{1}{1+\sigma}\Bigr) = \infty,\qquad
\lim_{\beta\to\beta_2^-} \lambda_i\Bigl(\frac{1}{1+\sigma}\Bigr)= 0,
\]
and hence, we understand why the result cannot be further extended for $\gamma\ge \pi$. 

From now on, we will assume that $0<\gamma<\pi/2$, in order to get a simple expression of the geodesic and to obtain neater results, for instance removing the appearance of the $\arctantwo$ function. 
We should note that in this case, $\beta$ should be in the interval $(\beta_1, 0) \cup (0, \beta_2)$ where
\begin{equation} \label{eq:beta1beta2}
\beta_1 := - \pi \frac{\sqrt{\pi^2n^2+4nd^2}+\pi n}{2nd^2}, \qquad \beta_2 := \pi \frac{\sqrt{\pi^2n^2+4nd^2} -\pi n}{2nd^2}.
\end{equation}

From Lemma~\ref{thm:lemma} we can obtain a simple expression for the geodesics, by assuming that $0<\gamma<\pi/2$.
\begin{Corollary} \label{thm:cor}
Let $D_A,D_B$ be two diagonal matrices in $\Pn$, $\beta\in(-\infty,0)\cup(0,1/n)$ and $\gamma:=\gamma_\beta(D_A,D_B)$ as in~\eqref{eq:gamma}. If $0<\gamma<\pi/2$, then there exists a unique geodesic $\wt P(t)$ on $\Mb$ joining $D_A$ and $D_B$ such that $\wt P(t)=\diag(\lambda_1(t),\ldots,\lambda_n(t))$, where
\begin{equation} \label{eq:scalar}
\lambda_i(t) = \biggl(\frac{(1-t)^2+\sigma^2t^2+2\sigma t (1-t)\cos\gamma}{\sigma^{2\alpha(t)}}\biggr)^{\frac{1}{n\beta}} \lambda_i(0)^{1-\alpha(t)}\lambda_i(1)^{\alpha(t)}, \qquad t \in [0,1],
\end{equation}
with $\sigma = \det(D_A^{-1}D_B)^{\beta/2}$ and
\[	
\alpha(t) = \frac{1}{\gamma}\arctan\biggl(\frac{t\sigma \sin\gamma}{1-t+t\sigma\cos\gamma}\biggr).
\]
\end{Corollary}

\begin{proof}
The key observation is that for $0<\gamma<\pi/2$, the geodesic in Lemma~\ref{thm:lemma} can be written as
\[
\lambda_i(t) = \lambda_i(0)\Bigl((1-t)(1-t+t\sigma\cos\gamma)(e^{\zeta_i})^{n\beta\alpha_1(t)} + t(t+(1-t)\frac{1}{\sigma}\cos\gamma)(e^{\zeta_i})^{-n\beta\alpha_2(t)}\Bigl(\frac{\lambda_i(1)}{\lambda_i(0)}\Bigr)^{n\beta}\Bigr)^{\frac{1}{n\beta}},
\]
where
\[
\alpha_1(t) = \frac{1}{\gamma}\arctan \phi_1(t),
\qquad
\alpha_2(t) = \frac{1}{\gamma}\arctan \phi_2(t),
\]
with
\[
\phi_1(t)=\frac{t\sigma\sin\gamma}{1-t+t\sigma\cos\gamma},
\qquad
\phi_2(t)=\frac{(1-t)\sin\gamma}{t\sigma+(1-t)\cos\gamma}.
\]
This follows from the fact that $\cos\gamma>0$ for $0<\gamma<\pi/2$, that in turn implies that both arguments of the $\arctantwo$ function are positive for $t\in(0,1)$ and we can substitute the latter function with the arctangent of the quotient.

We note that $\alpha_1(t)+\alpha_2(t)=1$. This result can be obtained by observing that the arguments of the arctangent $\phi_1(t)$ and $\phi_2(t)$ are such that $\phi_1(t)\phi_2(t)<1$, since $\cos\gamma>0$, because $0<\gamma<\pi/2$. The arctangent summation formula shows that $\arctan\phi_1(t)+\arctan\phi_2(t)=\arctan\tan\gamma=\gamma$, that in turn yields $\alpha_1(t)+\alpha_2(t)=1$. We further observe that
\[
\Bigl(\frac{\lambda_i(1)}{\lambda_i(0)}\Bigr)^{n\beta}=(e^{\zeta_i})^{n\beta}\sigma^2,
\]
or equivalently
\[
e^{\zeta_i}=\frac{\lambda_i(1)}{\lambda_i(0)}\sigma^{-2/(n\beta)}.
\]
Then, using these two observations, we obtain
\begin{align*}
\lambda_i(t) &= \lambda_i(0)\bigl((1-t)^2+2t(1-t)\sigma\cos\gamma+t^2\sigma^2\bigr)^{\frac1{n\beta}}(e^{\zeta_i})^{\alpha(t)} \\
&=\bigl((1-t)^2+2t(1-t)\sigma\cos\gamma+t^2\sigma^2\bigr)^{\frac1{n\beta}} \sigma^{-\frac{2\alpha(t)}{n\beta}}\lambda_i(0)^{1 - \alpha(t)} \lambda_i(1)^{\alpha(t)},
\end{align*}
where we have set $\alpha(t) = 1 - \alpha_1(t) = \alpha_2(t)$.
\end{proof}

Based on the reduction to the diagonal case we are now in a position to give the geodesic in the general case.
\begin{theorem} \label{thm:mainAB}
Let $A,B\in\Pn$. If $A$ and $B$ are linearly independent, then setting $\gamma:=\gamma_\beta(A,B)$, as in~\eqref{eq:gamma}, we have that for $\beta\in(\beta_1,0)\cup(0,\beta_2)$, where $\beta_1$ and $\beta_2$ are defined in~\eqref{eq:beta1beta2}, there exists a unique geodesic on $\Mb$ joining $A$ and $B$ and whose explicit expression is 
\begin{equation} \label{eq:maing}
G_\beta(A,B,t)=\eta(t)(A\#_{\alpha(t)} B)=\eta(t)A(A^{-1}B)^{\alpha(t)},\quad t\in[0,1],
\end{equation}
where
\begin{equation} \label{eq:alphaeta}
\alpha(t):=\frac{1}{\gamma}\arctan\Bigl(\frac{t\sigma\sin\gamma}{1-t+t\sigma\cos\gamma}\Bigr),\quad
\eta(t):=\Bigl(\frac{(1-t)^2+2t(1-t)\sigma \cos\gamma+t^2\sigma^2}{\sigma^{2\alpha(t)}}\Bigr)^{\frac{1}{n\beta}}
\end{equation}
with $\sigma=\det(A^{-1}B)^{\beta/2}$.

If $A$ and $B$ are linearly dependent, then for $\beta\in(-\infty,0)\cup(0,\frac{1}{n})$, there exists a unique geodesic on $\Mb$ joining $A$ and $B$ and whose explicit expression is 
\begin{equation} \label{eq:maing2}
G_\beta(A,B,t) = ((1-t)A^{\frac{n\beta}{2}}+tB^{\frac{n\beta}{2}})^{\frac{2}{n\beta}}, \qquad t \in [0,1].
\end{equation}
\end{theorem}

\begin{proof}
When $A$ and $B$ are linearly dependent we can use Corollary \ref{thm:lindep}. Assuming that the two matrices are independent, we can use the argument of Section \ref{sec:diag} and in order to obtain a geodesic $G_\beta(A,B,t)$ on $\Mb$ joining $A$ and $B$, it is enough to find a geodesic $\wt P(t)=\diag(\lambda_1(t),\ldots,\lambda_n(t))$ joining the two diagonal matrices $I=M^TAM$ and $D=M^TBM$ and then apply the inverse congruence to $\wt P(t)$. If $\beta\in(\beta_1,0)\cup(0,\beta_2)$, that is $0<\gamma<\pi/2$, then the geodesic is obtained using Corollary~\ref{thm:cor}, with $D_A=I$ and $D_B=D=\diag(\mu)$, where  $\mu=[\mu_1,\ldots,\mu_n]$, is the vector of the eigenvalues of $A^{-1}B$. 

Setting $\eta(t)$ as in~\eqref{eq:alphaeta} and using the expression provided in~\eqref{eq:scalar} and the fact that $\lambda_i(0)=1$, for $i=1,\ldots,n$, we obtain $\lambda_i(t)=\eta(t)\mu_i^{\alpha(t)}$, and thus
\begin{multline} \nonumber
G_\beta(A,B,t)=M^{-T}\wt P(t)M^{-1} = \eta(t)M^{-T}M^{-1}M\diag(\mu_1^{\alpha(t)},\ldots,\mu_n^{\alpha(t)})M^{-1} \\
= \eta(t)M^{-T}M^{-1}\bigl(M\diag(\mu_1,\ldots,\mu_n)M^{-1} \bigr)^{\alpha(t)}=\eta(t)A(A^{-1}B)^{\alpha(t)}=\eta(t)(A\#_{\alpha(t)}B),
\end{multline}
where the third equality is obtained by using the commutativity of primary matrix functions with similarities (applied to the fractional power of matrices).
\end{proof}

Theorem~\ref{thm:main}, provides the restrictions on $\beta$ guaranteeing that a geodesic connecting two given matrices $A$ and $B$ exists. Now, for a given $\beta$, we provide conditions on $A$ and $B$ that ensure the existence of a geodesic joining them.
\begin{theorem} \label{thm:main}
Let $A,B\in\Pn$, $\beta\in(-\infty,0)\cup(0,1/n)$, and $\gamma:=\gamma_\beta(A,B)$ as in~\eqref{eq:gamma}. If $\gamma<\pi/2$, then there exists a unique geodesic $G_\beta(A,B,t)$ on $\Mb$ joining $A$ and $B$ and whose explicit expression, for $\gamma\ne 0$, is given in~\eqref{eq:maing}, while, for $\gamma=0$, is given in~\eqref{eq:maing2}.
\end{theorem}

\begin{remark}\rm
Note that $\tan (\gamma \alpha(t)) = \frac {t \sigma \sin \gamma}{1-t+t \sigma \cos \gamma}$. Therefore, $\gamma \alpha(t)$ can be interpreted geometrically as the angle of a right triangle with adjacent side of length $1-t+t \sigma \cos \gamma$ and opposite side of length $t \sigma \sin \gamma$ (see Fig.~\ref{fig:triangle}). Let $\ell(t)$ denotes the length of the hypotenuse side, i.e.,
\[
\ell(t) := \sqrt{(1-t+t \sigma \cos \gamma)^2 + (t \sigma \sin \gamma)^2}.
\]
Then, as $\alpha'(t) = \frac{\sigma \sin \gamma}{\gamma \ell(t)^2}$, the angle $\gamma \alpha(t)$ is monotonically increasing from 0 to $\gamma$ as $t$ varies from 0 to 1.

\begin{center}
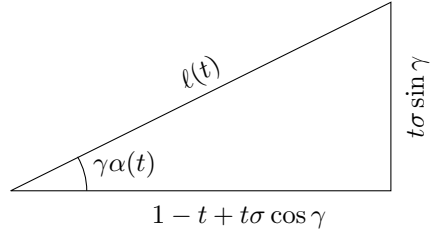
\begin{figure}[ht] 
\begin{center}
\def\t{-30}
\begin{tikzpicture}[scale=5]
\draw (0,0) -- (1,0) -- (1,0.5) -- (0,0);
\draw (0.2,0) arc (0:30:0.18);
\node at (0.6,-0.07) {\footnotesize $1-t+t\sigma\cos\gamma$};
\node[cm={0,1,-1,0,(0,0)}] at (1.07,0.25) {\footnotesize $t\sigma\sin\gamma$};
\node[cm={cos(\t),-sin(\t),sin(\t),cos(\t),(0,0)}] at (0.5,0.31) {\footnotesize $\ell(t)$};
\node at (0.3,0.07) {\footnotesize $\gamma\alpha(t)$};
\end{tikzpicture}
\end{center}
\caption{Geometric interpretation of $\alpha(t)$.}
\label{fig:triangle}
\end{figure}
\end{center}

We further observe that the function $\eta(t)$ can then be compactly written as
\[
\eta(t) = \left(\frac{\ell(t)}{\sigma^{\alpha(t)}} \right)^{\frac{2}{n\beta}}.
\]
\end{remark}  

\section{A new power mean for positive definite matrices} \label{sec:mean}

The geodesic given in Theorem~\ref{thm:mainAB} can be seen as a new possible generalization of the power mean to matrices. While for linearly dependent matrices the expression for $G_\beta(A,B,t)$ coincides with the ``power-Euclidean'' mean
\[
P_p(A, B, t) := ((1-t) A^p + t B^p)^{1/p},
\]
for $p=n\beta/2$ and with the power mean of matrices defined by Lim and P\'alfia \cite{limpalfia}, that is
\[
Q_p(A, B, t) := A \#_{1/p}((1-t) A + t (A\#_p B)),
\]
for linearly independent matrices it is different. This is easily seen by considering $I$ and a diagonal matrix $D=\diag(d_{11},\ldots,d_{nn})$, that is not a multiple of $I$. For $I$ and $D$ the power-Euclidean and the power mean of parameter $p=n\beta/2$ coincide and the diagonal entries are $(1-t+td_{ii}^p)^{1/p}$, while a diagonal entry of the proposed mean is obtained by using all diagonal entries of $D$.
\begin{example} \label{ex:1}\rm
Let $A=I$ and $B=\begin{bmatrix} 1 & 0\\0 & 2\end{bmatrix}$ and let $\beta=-1$, $p=-1$, $t=0.5$,
\[
P_p(A,B,t) = Q_p(A,B,t)=\begin{bmatrix} 1 & 0\\ 0 & 4/3\end{bmatrix},\qquad
G_\beta(A,B,t) \approx \begin{bmatrix} 1.01995 & 0\\0 & 1.35889\end{bmatrix},
\]
while the latter has been obtained numerically approximating the result to $6$ significant digits. 
\end{example}

Even if the proposed mean does not reduce to the straightforward power mean for commuting matrices, it fulfills a number of interesting properties of a power mean.

Before listing the properties, we recall here the following properties of the weighted geometric mean that will be used frequently in the sequel. It is remarkable that the $t$-weighted mean $G_\beta(A, B, t)$ of $A$ and $B$ \textcolor{blue}{is} a scalar multiple of the $\alpha(t)$-weighted geometric mean of $A$ and $B$. This further exhibits the intimate relation between the proposed mean and the geometric mean. 

\begin{Lemma} \label{thm:lemsharp}
Let $A,B\in\Pn$ and let $A\#_tB=A(A^{-1}B)^t$, for $t\in[0,1]$, be their weighted geometric mean. We have:
\begin{enumerate}
\item \label{item:sharp1} $(M^TAM)\#_t (M^TBM)=M^T(A\#_tB)M$, for $M$ invertible and $t\in[0,1]$;
\item \label{item:sharp2} $A\#_tB=B\#_{1-t}A$, for $t\in[0,1]$;
\item \label{item:sharp3} $A^{-1}\#_t B^{-1}=A^{-1}(A\#_{1-t}B)B^{-1}=B^{-1}(A\#_{1-t}B)A^{-1}$, for $t\in[0,1]$;
\item \label{item:sharp4} $(aA)\#_t (bB)=a^{1-t}b^t(A\#_tB)$, for $a,b>0$ and $t\in[0,1]$.
\end{enumerate}
\end{Lemma}

The quantities $\sigma$, $\gamma$, and the functions $\alpha(t)$, $\eta(t)$ appearing in Theorem~\ref{thm:mainAB} depend on $A$ and $B$. When needed, we will use the notation $\sigma(A,B)$, $\gamma(A,B)$, $\alpha(A,B,t)$, $\eta(A,B,t)$, to explicit this dependence. We give in the following lemma some of their symmetry and invariance properties that will be useful in the proof of the properties of the mean $G_\beta(A,B,t)$. 
\begin{lemma} \label{lem:alpha-eta-sigma}
The functions  $\sigma(A,B)$, $\gamma(A, B)$, $\alpha(A,B,t)$, $\eta(A,B,t)$ satisfy the following properties:
\begin{align*}
& \sigma(A, B) = \sigma(B, A)^{-1} = \sigma(B^{-1}, A^{-1}) = \sigma(M^TAM,M^TBM), \\
& \gamma(A, B) = \gamma(B, A) = \gamma(B^{-1}, A^{-1}) = \gamma(M^TAM,M^TBM), \\
& \alpha(A, B, t) = 1-\alpha(B, A, 1-t) = 1-\alpha(A^{-1},B^{-1}, 1-t) = \alpha(M^TAM,M^TBM, t), \\
& \eta(A, B, t) = \eta(B, A, 1-t) = \eta(A^{-1},B^{-1}, 1-t) = \eta(M^TAM,M^TBM, t),
\end{align*}
for any invertible matrix $M \in \R^{n\times n}$ and $t\in[0,1]$. Furthermore, for $t\in[0,1]$ and $a, b >0$, we have
\begin{align*}
  & \sigma(aA,bB) = (b/a)^{n\beta/2} \sigma(A,B), \\
  & \gamma(aA,bB) = \gamma(A,B), \\
  & \alpha(aA,bB,t) = \alpha(A,B,\wt t), \\
  & \eta(aA,bB,t) =  \frac{((1-t)a^{n\beta/2}+tb^{n\beta/2})^{2/(n\beta)}}{a^{1-\alpha(A,B,\wt t)}b^{\alpha(A,B,\wt t)}} \eta(A,B,\wt t),
\end{align*}
where $q=(b/a)^{n\beta/2}$ and 
$\wt t=\frac{tq}{1-t+tq}$.
\end{lemma}

\begin{proof}
The first set of properties follows directly from the definitions of $\sigma$, $\gamma$, $\alpha(t)$, $\eta(t)$ (and from the proof of Corollary~\ref{thm:cor} for $\alpha(t)$).

For the second set of properties, we have  
\[
\frac{\wt t\sigma(A,B)\sin \gamma(A,B)}{1-\wt t+\wt t\sigma(A,B)\cos \gamma(A,B)} = \frac{t\sigma(aA,bB)\sin\gamma(aA,bB)}{1-t+t\sigma(aA,bB)\cos\gamma(aA,bB)},
\]
from which it follows that $\alpha(A,B,\wt t)=\alpha(aA,bB,t)$ and
\[
\frac{(1-\wt t)^2+2\wt t\sigma(A,B)\cos \gamma(A,B)+t^2\sigma(A,B)^2}{\sigma(A,B)^{2\alpha(A,B,\wt t)}}
= \frac{(1-t)^2+2t\sigma(aA,bB)\cos\gamma(aA,bB)+t^2\sigma(aA,bB)^2}{\sigma(aA,bB)^{2\alpha(aA,bB,t)}q^{-2\alpha(aA,bB,t)}(1-t+tq)^2},
\]
from which we get
\[
\eta(A,B,\wt t)=\eta(aA,bB,t)\frac{(b/a)^{\wt \alpha}a}{((1-t)a^{n\beta/2}+tb^{n\beta/2})^{2/(n\beta)}},
\]
where $\wt \alpha:=\alpha(A,B,\wt t)=\alpha(aA,bB,t)$.
\end{proof}

We will use the properties of the weighted geometric mean of two matrices listed in Lemma~\ref{thm:lemsharp} and the results of Lemma~\ref{lem:alpha-eta-sigma} to prove some properties of the proposed mean $G_\beta(A,B,t)$, then we will discuss some asymptotic properties.

\begin{theorem} \label{thm:properties}
Let $\beta\in(-\infty,0)\cup(0,1/n)$, $A,B\in\Pn$ such that $0<\gamma<\pi/2$, with $\gamma:=\gamma_\beta(A,B)$ as in~\eqref{eq:gamma}. The following properties hold.
\begin{enumerate}
\item \label{item:M2} $G_\beta(M^TAM,M^TBM,t)=M^TG_\beta(A,B,t)M$, for any $M\in\R^{n\times n}$ invertible and $t\in[0,1]$.
\item \label{item:change2} $G_\beta(A,B,t)=G_\beta(B,A,1-t)$, for $t\in[0,1]$.
\item \label{item:inversion2}
$G_\beta(A^{-1},B^{-1},t)=A^{-1}G_\beta(A,B,1-t)B^{-1}=A^{-1}G_\beta(B,A,t)B^{-1}$\\
$\hphantom{G_\beta(A^{-1},B^{-1},t)}  =B^{-1}G_\beta(B,A,t)A^{-1} =B^{-1}G_\beta(A,B, 1-t)A^{-1}$, for $t\in[0,1]$.

\item \label{item:inversion3}
$\frac{G_\beta(A^{-1},B^{-1},t)}{\eta(A^{-1},B^{-1},t)} = \big( \frac{G_\beta(A,B,t^\#)}{\eta(A,B,t^\#)} \big)^{-1}$, where for $t\in[0,1]$, $t^\#$ is such that $\alpha(A,B,t^\#) = \alpha(A,B,1-t)$.

\item \label{item:hom2} $G_\beta(aA,bB,t)=
((1-t)a^{n\beta/2}+tb^{n\beta/2})^{\frac{2}{n\beta}}G_\beta\Bigl(A,B,\frac{t b^{n\beta/2}}{(1-t)a^{n\beta/2}+t b^{n\beta/2}}
\Bigr)$ \\
$\hphantom{G_\beta(aA,bB,t)}=
((1-t)a^{n\beta/2}+tb^{n\beta/2})^{\frac{2}{n\beta}}G_\beta\Bigl(B,A,\frac{(1-t) a^{n\beta/2}}{(1-t)a^{n\beta/2}+t b^{n\beta/2}}
\Bigr)$, for $a,b$ positive reals and~{$t\in[0,1]$}.
\item\label{item:sharp2'} $G_\beta(A,B,1/(1+\sigma))=
\bigl(\frac{2\sqrt{\sigma}\cos(\gamma/2)}{1+\sigma}\bigr)^{\frac{2}{n\beta}}
A\#_{1/2}B$, where $\sigma=\det(A^{-1}B)^{\beta/2}$.
\end{enumerate}
\end{theorem}

\begin{proof}
Using the results of Lemma~\ref{lem:alpha-eta-sigma} we have:
\begin{enumerate}
\item 
\begin{align*}	
G_\beta(M^TAM,M^TBM,t)&=\eta(M^TAM,M^TBM,t)((M^TAM)\#_{\alpha(M^TAM,M^TBM,t)}(M^TBM))\\
&=\eta(A,B,t)M^T(A\#_{\alpha(A,B,t)}B)M=M^TG_\beta(A,B,t)M.
\end{align*}

\item 
\[
G_\beta(B,A,1-t)=\eta(A,B,t)(B\#_{1-\alpha(A,B,t)}A)=\eta(A,B,t)(A\#_{\alpha(A,B,t)}B)=G_\beta(A,B,t).
\]

\item
\[
G_\beta(A^{-1},B^{-1},t)=\eta(B,A,t)(A^{-1}\#_{\alpha(B,A,t)}B^{-1})=
A^{-1}(\eta(B,A,t) (A\#_{\alpha(B,A,t)} B))B^{-1},
\]
from which the proof follows using property~\ref{item:sharp3} in Lemma~\ref{thm:lemsharp}.

\item Since $G_\beta(A^{-1},B^{-1},t) = A^{-1} G_\beta(A,B,1-t) B^{-1}$, we get
\begin{align*}
G_\beta(A^{-1},B^{-1},t) &= A^{-1} \eta(A,B,1-t) A(A^{-1}B)^{\alpha(A,B,1-t)} B^{-1} \\
&= \eta(B,A,t) (A^{-1}B)^{1 - \alpha(B,A,t)} B^{-1} = \eta(A^{-1},B^{-1},t) (B^{-1}A)^{\alpha(B,A,t)} A^{-1} \\
&= \eta(A^{-1},B^{-1},t) \big( A(A^{-1}B)^{\alpha(A,B,1-t)} \big)^{-1}
= \eta(A^{-1},B^{-1},t) \Big( \frac{G_\beta(A,B,t^{\#})}{\eta(A,B,t^{\#})} \Big)^{-1}.
\end{align*}

\item With the notation of Lemma \ref{lem:alpha-eta-sigma} and its proof, we get
\[
G_\beta(A,B,\wt t)=\eta(A,B,\wt t)(A\#_{\wt \alpha}B)=
\frac{\eta(aA,bB,t)a^{1-\wt \alpha}b^{\wt \alpha} (A\#_{\wt \alpha} B)}{((1-t)a^{n\beta/2}+tb^{n\beta/2})^{2/(n\beta)}}=
\frac{G_\beta(aA,bB,t)}{((1-t)a^{n\beta/2}+tb^{n\beta/2})^{2/(n\beta)}},
\]
where we have used also item \ref{item:sharp4} of Lemma \ref{thm:lemsharp}.

\item For $t=1/(1+\sigma)$, we have $\alpha(t)=1/2$ and the formula is obtained through a simple manipulation using the cosine duplication formula.
\end{enumerate}
\end{proof}

We note that some of these properties are shared with the power mean of Lim and P\'alfia, for instance, properties 1, 2 and 5, holds also for the latter, see Proposition 3.5 of \cite{limpalfia}. A big difference is the behavior with respect to inversion, while the mean by Lim and P\'alfia verifies the inversion property
\[
Q_p(A, B, t) = (Q_{-p}(A^{-1}, B^{-1}, t))^{-1},
\]
the proposed mean fulfills properties 3 and 4 of Theorem \ref{thm:properties} which are (in general) different than the inversion property. Moreover, it is unclear whether the new mean has some property related to the natural order of positive-definite matrices.

Next, we provide some asymptotic properties of the new mean.
\begin{theorem}
For all $A,B\in\Pn$ we have
\[
\lim_{\beta\to 0} G_{\beta}(A, B, t) = A \#_t B.
\]
\end{theorem}

\begin{proof}
With no loss of generality we can assume that $A$ and $B$ are linearly independent. For $\beta$ sufficiently small $G_\beta(A,B,t)$ is well defined so that we can take the limit.

Since both $G_\beta(A,B,t)$ and the weighted geometric mean commute with congruences, we can assume that $A=I$ and $B=D=\diag(d_1,\ldots,d_n)$ and prove that $G_\beta(I,D)\to I\#_t D=D^t$. In this case $G_\beta(A,B,t)=\diag(\lambda_1(t),\ldots,\lambda_n(t))$, and it is enough to prove that $\lambda_i(t)\to d_i^t$ for $i=1,\ldots,n$.

Setting $\omega_0=\ln\det(D)$, we observe that, as $\beta\to 0$,
\[
\sigma=1+\omega_0\beta/2+o(\beta),\qquad \gamma=\omega_1 \beta+o(\beta),
\]
where $\omega_1$ is a nonzero constant. From this we deduce that for $\beta\to 0$
\[
t\sigma\sin\gamma=t\gamma+o(\beta),\qquad 1-t+t\sigma\cos\gamma=1+o(1),\qquad
   \alpha(t) = t+o(1),\qquad
d_i^{\alpha}=d_i^t+o(1),
\]
and, analogously
\[
(1-t)^2+2t(1-t)\sigma\cos\gamma+t^2\sigma^2=1+\omega_0\beta t+o(\beta),\qquad
\sigma^{2\alpha}=1+\omega_0\beta t+o(\beta),\qquad \eta(t)=1+o(\beta).
\]
Finally, since $\lambda_i(0)=1$ and $\lambda_i(1)=d$, we have
\[
\lambda_i(t)=\lambda_i(0)^{1-\alpha(t)}\lambda_i(1)^{\alpha(t)}\eta(t)=d_i^t+o(1).
\]
that is what we wanted to prove.
\end{proof}

In Theorem~\ref{thm:mainAB}, there is a distinction between the case in which $A$ and $B$ are linearly dependent and the case in which they are not, and the form of the geodesic is different. One might ask whether this mean is continuous with respect to $A$ and $B$, for this reason, we show that, for a given $t$ and $\beta$, when the two matrices approach a couple of linearly dependent matrices, the formula~\eqref{eq:maing} tends to~\eqref{eq:maing2}.

\begin{theorem} \label{thm:gamma0}
Let $M\in\Pn$, $N_0,N_1\in\Sn$ and $\omega_0,\omega_1$ be positive constants. Let $A(\tau)=\omega_0 M+\tau N_0$ and $B(\tau)=\omega_1 M+\tau N_1$, be two curves on $\Pn$, defined in a neighborhood of $0$ and such that $A(\tau)$ and $B(\tau)$ are linearly independent for $\tau\ne 0$. We have that
\[
\lim_{\tau\to 0} G_\beta(A(\tau),B(\tau),t)=G_\beta(A(0),B(0),t),
\]
for $\beta\in(-\infty,0)\cup(0,1/n)$.
\end{theorem}

\begin{proof}
There exist continuous functions $\wt \mu_1(\tau),\ldots,\wt \mu_n(\tau)$ that are the eigenvalues of $\wt A(\tau)^{-1}\wt B(\tau)$, where as usual $\wt A(\tau)=\det(A(\tau))^{-1/n}A(\tau)$, $\wt B(\tau)=\det(B(\tau))^{-1/n}B(\tau)$.
Hence, the function 
\[
\gamma(\tau):=\gamma_\beta(A(\tau),B(\tau)) = \frac{|\beta|}{2\sqrt{1/n-\beta}}\Bigl(\sum_i \ln^2\wt \mu_i(\tau)\Bigr)^{1/2}
\]
is continuous and $\lim_{\tau\to 0}\gamma(\tau)=0$, since $A(0)$ and $B(0)$ are linearly dependent. There exists a neighborhood of $0$ such that $\gamma(\tau)<\pi/2$ and the function $G_\beta(A(\tau),B(\tau),t)$ is well defined in this neighborhood.

Since the mean commutes with congruences, without loss of generality, we can assume that $M=I$ and $\omega_0=1$. We have that 
\[
\sigma(\tau) = \det(A(\tau)^{-1}B(\tau))^{\beta/2},
\]
is continuous and $\lim_{\tau \to 0}\sigma(\tau)=\omega_1^{n\beta/2} =:\omega$.

As $\tau$ approaches 0, we have the power series expansions
\[
\sigma(\tau) = \omega + \frac{\omega \beta}2 \trace(N) \tau  + o(\tau),\qquad \wt \mu_i(\tau) = 1 + \wh \mu_i \tau + o(\tau), 
\]
where $N:=\frac{N_1}{\omega_1}-N_0$ and $\wh \mu_i$ are the eigenvalues of $N-\frac{1}{n}\trace(N)I$, and moreover
\[
\gamma(\tau) = \frac{|\beta|}{2\sqrt{1/n-\beta}}\bigl\| N-\trace(N)I/n\bigr\|_F\tau+o(\tau),
\]
\[
\alpha=\frac{t\omega}{1-t+t\omega}+o(1).
\]
Thus
\[
\bigl((1-t)^2+2t(1-t)\sigma\cos\gamma+t^2\sigma^2\bigr)^{\frac{2}{n\beta}}=(1-t+t\omega_1^{\frac{n\beta}{2}})^{\frac{2}{n\beta}}+o(1),\qquad \sigma^{\frac{2\alpha}{n\beta}} =\omega_1^\alpha+o(1),
\]
and
\begin{align*}
A\#_\alpha B &=(I+\tau N_0)\bigl((I+\tau N_0)^{-1}(\omega_1 I+\tau N_1)\bigr)^\alpha
=(I+\tau N_0)(\omega_1 I+\tau (N_1-\omega_1 N_0)+o(\tau))^\alpha \\
&=\omega_1^\alpha(I+\tau N_0+\tau\alpha N+o(\tau)),
\end{align*}
from which we finally get
\[
\lim_{\tau \to 0} G_\beta(A(\tau),B(\tau),t)=(1-t+t\omega_1^{\frac{n\beta}{2}})^{\frac{2}{n\beta}}I=G_\beta(I,\omega_1 I).
\]
By reverting the congruence, the result follows.
\end{proof}

\section{Riemannian distance} \label{sec:dist}

We give an explicit expression for the Riemannian distance in $\Mb$, in terms of the determinants of $A$ and $B$ and the classical Riemannian distance in $\mathcal{M}_{n}^0$. 

\begin{theorem}
The distance associated with the Riemannian metric $\Mb$ between two positive-definite matrices $A$ and $B$, such that $0<\gamma <\pi/2$, with $\gamma:=\gamma_\beta(A,B)$ of \eqref{eq:gamma}, is given by
\begin{equation} \label{eq:beta-dist}
\dis_{\beta}(A, B) = \frac{2 \sqrt{1/n - \beta}}{|\beta|} \Bigl( \left(\det(A)^{\beta/2} - \det(B)^{\beta/2} \right)^2 + 4(\det(A) \det(B))^{\beta/2} \sin^2 \frac{\gamma}{2} \Bigr)^{1/2}.
\end{equation}
\end{theorem}

\begin{proof}
We rely on the proof of Lemma \ref{thm:lemma} to get some useful expressions involving the geodesics $P(t)$ and $G(t)=P(t)^{-1}P'(t)$.

Using~\eqref{eq:Gi} and the fact that $\sum_{i=1}^n a_i = 0$, we obtain
\[
\trace (G(t)) = \frac{2 h(t) h'(t)}{\beta(1 + h(t)^2)},
\]
where $h(t) = (1-t) \phi + t \psi$, with $\phi$ and $\psi$ as in \eqref{eq:phi-psi}. Furthermore, by noting that $\sum_{i=1}^n a_i^2=2n(1-n\beta)a^2$, it follows from~\eqref{eq:Gi} that
\[
\trace (G(t)^2) = \frac{4(\phi - \psi)^2}{n\beta^2(1 + h(t)^2)^2} (h(t)^2 + 1 - n \beta).
\]
From equation \eqref{eq:diag-geod}, we obtain that
\begin{align} \label{eq:detP^beta}
\det(P(t))^\beta &= \prod_{i=1}^n \lambda_i(0)^\beta \Bigl(\frac{1+h(t)^2}{1+\phi^2}\Bigr)^{1/n}\exp\bigl(b_i(\arctan(h(t))-\arctan\phi)\bigl) \nonumber \\
& =\det(A)^\beta\frac{1+h(t)^2}{1 + \phi^2},
\end{align}
where we have used the fact that $\sum_i b_i=0$.

The ``square of the speed'' of the geodesic curve $P(t)$ is given by 
\begin{equation} \label{eq:dist1}
g_{P}^\beta(P',P') = \det(P(t))^{\beta} \left[\trace(G(t)^2) - \beta \trace^2(G(t)) \right] = 4\frac{\det(A)^\beta}{1+\phi^2}\frac{1/n - \beta}{\beta^2}(\psi - \phi)^2.
\end{equation}
The latter is constant and hence $t$ is the arclength of $P(t)$. Using the expressions for $\phi$ and $\psi$, we obtain the identity
\[
\det(A)^\beta\frac{(\psi-\phi)^2}{1+\phi^2} = (\det(A)\det(B))^{\beta/2}\bigl(\sigma+\frac{1}{\sigma}-2\cos\gamma\bigr),
\]
that in turn, as $\sigma=\det(B)^{\beta/2}/\det(A)^{\beta/2}$, provides the formula for the distance

\begin{equation} \label{eq:dist0}
\begin{split}
\dis_{\beta}(A, B)^2
&= \frac{4(1/n - \beta)}{\beta^2} \left(\det(A)^\beta + \det(B)^\beta - 2(\det(A) \det(B))^{\beta/2} \cos \gamma \right), \\
&= \frac{4(1/n - \beta)}{\beta^2} \left( \left(\det(A)^{\beta/2} - \det(B)^{\beta/2} \right)^2 + 2(\det(A) \det(B))^{\beta/2} (1 - \cos \gamma) \right) \\
&= \frac{4(1/n - \beta)}{\beta^2} \left( \left(\det(A)^{\beta/2} - \det(B)^{\beta/2} \right)^2 + 4(\det(A) \det(B))^{\beta/2} \sin^2 \frac{\gamma}{2} \right).
\end{split}
\end{equation}
\end{proof}

The explicit formula \eqref{eq:beta-dist} involves the dimension $n$, the parameter $\beta$, the determinants of $A$ and $B$ and the classical Riemannian distance in $\mathcal{M}_n^{0}$ between $A$ and $B$, through $\gamma$, see Fig.~\ref{fig:triangle2} for a geometric interpretation.

\begin{center}
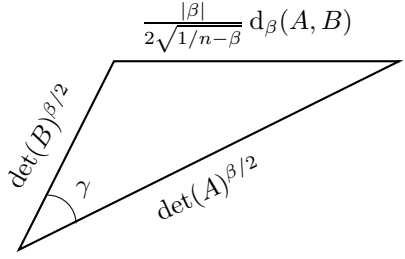
\begin{figure}[ht] 
\begin{center}
\begin{tikzpicture}[scale=5]
\def\t{33.7}
\def\s{60}  
\def\theta{25}
\draw[thick] (0,0) -- (1,0.5) -- (0.25,0.5) -- cycle;
\draw (0.15,0.07) arc (0:90:0.075);
\node[cm={cos(\theta), sin(\theta), -sin(\theta), cos(\theta), (0,0)}] at (0.5,0.16) {\footnotesize $\det(A)^{\beta/2}$};
\node[cm={cos(\s), sin(\s), -sin(\s), cos(\s),(0,0)}] at (0.06,0.3) {\footnotesize $\det(B)^{\beta/2}$};
\node at (0.6,0.59) {\footnotesize $\frac{|\beta|}{2\sqrt{1/n-\beta}} \dis_\beta(A,B)$};
\node at (0.17,0.16) {\footnotesize $\gamma$};
\end{tikzpicture}
\end{center}
\caption{Geometric interpretation of $\dis_\beta(A, B)$.}
\label{fig:triangle2}
\end{figure}
\end{center}

Note that when $\det A = \det B = \Delta$, then~\eqref{eq:beta-dist} reduces to
\[
\dis_{\beta}(A, B) = \frac{4\sqrt{1/n - \beta} }{|\beta|} \Delta^{\beta/2}\sin \frac{\gamma}{2}.
\]

As $\beta\to 0$, we have that $\gamma=\frac{1}{2}|\beta|\sqrt{n}\delta(\wt A,\wt B)+o(|\beta|)$ that yields the expansion
\[	
\dis_\beta(A,B)^2=\frac{4}{n\beta^2}\Bigl(\frac{\beta^2}{4}(\log\det(A)-\log\det(B))^2-
\frac{n\beta^2}{4}\delta(\wt A,\wt B)^2+o(\beta^2)\Bigr)=\delta(A,B)^2+o(1),
\]
that gives
\[
\lim_{\beta \to 0} \dis_{\beta}(A, B) = \delta(A, B),
\]
and the new distance tends to the Riemannian distance between $A$ and $B$.

Finally, we give in the following lemma an expression for the determinant of the point $\det P(t)$ along the geodesics as a function of $\det(A)$ and $\det(B)$.

\begin{lemma}
Along the geodesic curve $P(t)$ joining $A$ and $B$, the determinant $\Delta(t):= \det P(t)$ is given by
\begin{align*}
\Delta(t)^{\beta} &= (1 - t)^2 \det(A)^\beta + t^2 \det(B)^\beta + 2t(1 - t) \det(A)^{\beta/2} \det(B)^{\beta/2} \cos \gamma \\
&= \left( (1 - t) \det(A)^{\beta/2} + t \det(B)^{\beta/2} \right)^2 - 4t(1 - t) \det(A)^{\beta/2} \det(B)^{\beta/2} \sin^2 \frac{\gamma}2. 
\end{align*}
\end{lemma}

\begin{proof}
Along the geodesic curve $P(t)$ joining $A$ and $B$, we know that the ``square of the speed'',
\[
g_{P(t)}^{\beta}(G(t), G(t)) = \det (P(t))^\beta \left( \trace(G(t)^2) - \beta \trace^2 G(t) \right), 
\]
is constant and is equal to $D^2 := \dis_\beta^2(A,B)$, the square of the $\beta$-Riemannian distance between $A$ and $B$.

Therefore, by using Jacobi's formula, $\Delta'(t) / \Delta(t) = \trace(G(t))$, and the above we obtain
\[
\trace(G(t)^2) = D^2 \Delta(t)^{-\beta} + \beta \left(\frac{\Delta'(t)}{\Delta(t)}\right)^2.  
\]
Furthermore, upon derivation of Jacobi's formula we get
\[
\left(\frac{\Delta'(t)}{\Delta(t)} \right)' = \trace(G'(t)),
\]
which, with the help of \eqref{eq:ode-trG}, can be written as
\[
\left(\frac{\Delta'(t)}{\Delta(t)} \right)' = \frac{n \beta D^2}{2(1-n\beta)} \Delta(t)^{-\beta} - \beta \left(\frac{\Delta'(t)}{\Delta(t)} \right)^2,
\]
or, equivalently,
\[
\Delta(t)'' \Delta(t) + (\beta - 1) (\Delta(t)')^2 = \frac{n \beta D^2}{2(1-n\beta)} \Delta(t)^{2-\beta}.
\]
By integrating this second-order ODE subject to the conditions $\Delta(0) = \det A$ and $\Delta(1) = \det B$ we get the following expression for the determinant $\Delta(t)$:
\[
\Delta(t)^{\beta} = (1 - t) \det(A)^\beta + t \det(B)^\beta - t(1 - t) \frac{n D^2 \beta^2}{4(1 - n \beta)}.
\]
The result then follows by substituting the expression of $D^2$.
\end{proof}

When $A$ and $B$ are linearly dependent the above reduces to
\begin{equation*}
\Delta(t)^{\beta} = \left[ (1-t)\det(A)^{\beta/2} + t \det(B)^{\beta/2} \right]^2.
\end{equation*}

\section{Conclusions}\label{sec:conc}

Using the Hessian of the power potential function $(1-\det(X)^\beta)/\beta$ we have derived a Riemannian metric on the cone of positive definite matrices. We were able to find an explicit expression for the geodesics joining two matrices, when $\beta$ is sufficiently small or when the two matrices are sufficiently near to a couple of linearly dependent matrices.

The geodesic has been interpreted as a power mean of positive definite matrices, because of its properties, while it is different from the power means defined in the literature so far, and thus it represents a new mathematical object. 

Some open problems regarding the new geometry have been left: it is not clear how the new power mean interplay with the ordering of positive definite matrices, i.e. monotonicity. Moreover, extensive numerical tests suggest that the new geometry yields a nonpositive curvature, while a formal proof of this property is still lacking.

Finally, the relation of the new mean with information geometry and Tsallis statistics make it potentially useful in application where data matrices need to be averaged. An implementation of the mean can be found at MATLAB Central repository \url{https://tinyurl.com/tzkv3sqh}

Investigating the scope of applications of the new mean and extending this mean to more than two matrices is an object of future work.

\appendix

\renewcommand{\theequation}{A.\arabic{equation}}

\section*{Appendix A. Proof of Lemma \ref{thm:lemma}}\label{sec:app1}

\begin{proof}
The set of positive-definite diagonal matrices is a totally-geodesic submanifold of $\Mb$ (see Corollary \ref{thm:tot}), thus, we can assume that the geodesic $P(t)$ we are looking for is diagonal. 

Under the hypotheses on $\gamma$, we will solve equation~\eqref{eq:GeodEq}, with $G(t)$ diagonal for $t\in[0,1]$ and with suited boundary conditions.

Let $P(t) = \diag(\lambda_1, \ldots, \lambda_n)$ and $G(t) = \diag(\gamma_1, \ldots, \gamma_n)$, where $\gamma_i = {\lambda'_i}/{\lambda_i}$. Furthermore, set $\alpha(t)=\frac{1}{n}\trace(G(t))$ and $\wt{G}(t)=G(t)-\frac{1}{n}\trace(G(t))I=\diag(\nu_1, \ldots, \nu_{n})$ where $\nu_i = \gamma_i - \alpha$, $i=1, \ldots, n$. Observing that  $\nu_1 + \cdots + \nu_n=0$, the boundary-value problem to be solved can be written as (see~\eqref{eq:GeodEq2})
\begin{equation} \label{eq:GeodEq3}
\left\{\begin{array}{l} 
 \alpha' = - n \beta \left( \frac{\alpha^2}2 - \frac{1}{n(1 - n \beta)} \sum_{i=1}^{n-1} \left(\nu_i^2 + \sum_{j<i} \nu_i \nu_j \right) \right), \\
 \nu_i' =  - n \beta \alpha \nu_i, \quad i=1,\ldots, n-1,\\
 \lambda'_i/\lambda_i = \nu_i+\alpha,\quad i=1,\ldots,n,\\
 \lambda_i(0)=a_{ii},\quad \lambda_i(1)=b_{ii},\qquad i=1,\ldots,n. 
\end{array}\right.
\end{equation}
The key observation for getting an explicit solution of the above system is to consider the corresponding initial-value problem in a right neighborhood of $0$ and extend the solution so that it satisfies the boundary conditions. 

The case in which $\nu_i(0)=0$ for each $i$ can be ruled out by observing that, setting $\nu_i(t)\equiv 0$, the solution of \eqref{eq:GeodEq3} is obtained by
\[
\left\{
\begin{array}{l}
\alpha'=-\frac{n\beta}{2}\alpha^2,\\
\lambda'_i/\lambda_i = \alpha,\qquad i=1,\ldots,n,\\
\lambda_i(0)=a_{ii},\quad \lambda_i(1)=b_{ii},\qquad i=1,\ldots,n,
\end{array}
\right.
\]
whose solution is $\lambda_i(t) = \Bigl( (1 - t) a_{ii}^{\frac{n\beta}2} + t b_{ii}^{\frac{n\beta}2} \Bigr)^{\frac{2}{n\beta}}$ (see Section \ref{sec:special}). The condition $\lambda'_i/\lambda_i=\alpha$ for $i=1,\ldots,n$, implies that the quotient $b_{ii}/a_{ii}$ is constant and thus $D_A$ and $D_B$ are linearly dependent, that implies $\gamma=0$.

Without loss of generality, from now on, we assume that $\nu_\ell(0)\neq 0$ for some $\ell$. In a right neighborhood of $0$, for the indices such that $\nu_i(0)\ne 0$, we can write
\begin{equation} \label{eq:nu'/nu}
\frac{\nu_i'}{\nu_i} = \frac{\nu_\ell'}{\nu_\ell},
\end{equation}
that yields $\nu_i=a_i\nu_\ell$, for some constant $a_i$ obtained by integrating~\eqref{eq:nu'/nu}. In general, setting $a_i=0$ when $\nu_i(0)=0$, we can write
\[
\nu_i = a_i \nu_\ell, \quad 1 \le i \le n,
\]
where $a_\ell=1$ and $a_n=-\sum_{i=1}^{n-1}a_i$. By the uniqueness of solution of the initial value problem, we can further assume that $\nu_\ell(t)\ne 0$ for any $t$ in the domain of definition,  and moreover, that $\nu_\ell(t)>0$, since $\sum_i\nu_i=0$.

This reduces the problem to the system of two ordinary differential equations
\begin{equation} \label{eq:GeodEq4}
\begin{cases} 
\alpha' = - n \beta \left( \frac{\alpha^2}2 - \wh a \nu_\ell^2 \right), \\
\nu_\ell' = - n \beta \alpha \nu_\ell,
\end{cases}
\end{equation}
where
\[
\wh a = \frac{1}{2n(1 - n \beta)} \sum_{i=1}^n a_i^2.
\]

The solution of~\eqref{eq:GeodEq4} can be expressed as
\begin{subequations}
\begin{align}
& \alpha(t) = \frac1{n \beta} \frac{2 h(t) h'(t)}{1 + h(t)^2}, \\
& \nu_\ell(t) = \frac{\sqrt{2}}{n \beta a} \frac{h'(t)}{1 + h(t)^2},
\end{align}
\end{subequations}
where $h(t) = (1 - t) \phi + t \psi$ with $\phi$ and $\psi$ are integrating constants, while $a$ is one of the square roots of $\wh a$. Observe that we have assumed that $\nu_\ell(0)>0$ and this implies that $\psi\ne \phi$, and $\sign(\psi-\phi)\sign(a)=\sign(\beta)$, that will be assumed from now on.

It follows that the solution of the initial-value problem associated with~\eqref{eq:GeodEq3} is obtained by integrating
\begin{equation} \label{eq:Gi}
\frac{\lambda'_i}{\lambda_i} = \frac1{n\beta} \left(\frac{2 h(t) h'(t)}{1 + h(t)^2} +  \frac{\sqrt{2} a_i}{a} \frac{h'(t)}{1 + h(t)^2} \right),  \qquad i=1, \ldots, n,
\end{equation}
which yields the following formula for the geodesic curve
\begin{equation} \label{eq:diag-geod}
\lambda_i^{n\beta}(t) = \lambda_i^{n\beta}(0) \frac{1 + h(t)^2}{1 + \phi^2} \exp(b_i(\arctan h(t) - \arctan \phi)),
\end{equation}
for $i=1,\ldots,n$, where $b_i=\sqrt{2}a_i/a$ and $\lambda_i(0)=a_{ii}$. Notice that $\sum_{i=1}^n b_i=0$ and $b_\ell=\sqrt{2}/a$.

The constants ($\phi$, $\psi$, $b_1$, \ldots, $b_{\ell-1}$, $b_{\ell+1}$, \ldots, $b_n$) are to be determined so that the $n$ end conditions ($\lambda_i(1) = b_{ii}$, $i=1, \ldots, n$) are satisfied, together with $\sum_i b_i=0$.

Since $\nu_i(t)=\frac{\lambda_i'}{\lambda_i}-\frac{1}{n}\sum_{j=1}^n \frac{\lambda_j'}{\lambda_j}$, we have that, for $t$ in a right neighborhood of $0$,
\[
\int_0^{t}\nu_i(s)ds=\ln\frac{\lambda_i(t)}{\lambda_i(0)}-\frac{1}{n}\det(P(0)^{-1}P(t)).
\]
If we assume that the solution exists for $t=1$, we get
\[
\int_0^1\nu_i(s)ds=\ln\mu_i-\frac{1}{n}\ln\det(D_A^{-1}D_B)=\zeta_i,
\]
and since $\nu_i(t)=a_i\nu_\ell(t)$, we have that $a_i=\zeta_i/\zeta_\ell$, where $\zeta_i$ is defined in \eqref{eq:zeta}, and thus $b_i=\sign(a)2\sqrt{n(1-n\beta)}\zeta_i/\|\zeta\|$. Recall that we have assumed that $\nu_\ell(t)>0$ for each $t$, from which it follows that $\zeta_\ell>0$.

It remains to determine $\psi$ and $\phi$. From~\eqref{eq:diag-geod} it follows that, in order to have $\lambda_i(1)=b_{ii}$, it must be
\[
b_{ii}^{n\beta}=a_{ii}^{n \beta}\frac{1+\psi^2}{1+\phi^2}\exp\bigl(b_i(\arctan\psi-\arctan{\phi})\bigr).
\]
Therefore,
\begin{equation} \label{eq:bi}
b_i (\arctan \psi - \arctan \phi)=\beta \ln\Bigl(\frac{b_{ii}^n}{a_{ii}^n} \frac{\det D_A}{\det D_B}\Bigr) = n \beta \zeta_i,
\end{equation}
and, by observing that $\sum_i b_i=0$ and $\sum_i b_i^2=4n(1-n\beta)$, we get the two nonlinear equations on the unknowns $\phi$ and $\psi$
\begin{equation} \label{eq:phi-psi0}
\frac{1+\psi^2}{1+\phi^2}=\biggl(\prod_{i=1}^n \frac{b_{ii}}{a_{ii}}\biggr)^{\beta}=\biggl(\frac{\det D_B}{\det D_A }\biggr)^{\beta},\quad
(\arctan \psi - \arctan \phi)^2 = \frac{n^2 \beta^2 \|\zeta\|^2}{4n(1-n\beta)}=\gamma^2.
\end{equation}

Notice that the latter equation has no solution for $\gamma\ge \pi$ from which we get the hypothesis of the theorem. For $\gamma\neq\pi/2$, we take square roots on the second equation and apply the tangent function obtaining the system
\begin{equation}\label{eq:314}
	1+\psi^2=\sigma^2(1+\phi^2),\qquad \frac{|\psi-\phi|}{1+\psi\phi}=\tan\gamma,
\end{equation}
where $\sigma=\det(D_A^{-1}D_B^{-1})^{\beta/2}$ and the second equality is obtained using the tangent summation formula 
\[
\tan(|\arctan\psi-\arctan\phi|)=|\psi-\phi|/(1+\psi\phi).
\]

The real solutions of system \eqref{eq:314} are 
\begin{equation} \label{eq:phi-psi}
\phi = \frac{\sigma^{-1}-\cos\gamma}{\sin\gamma},\qquad \psi = \frac{\cos\gamma-\sigma}{\sin\gamma}
\end{equation} 
and $(-\phi,-\psi)$ and these two couples solve the system \eqref{eq:phi-psi0}. Moreover, by direct inspection, one shows that they solve the system also for $\gamma=\pi/2$.

We know that $\sign(a)\sign(\psi-\phi)=\sign(\beta)$, thus choosing the sign of $a$ imposes the choice of the couple. But, since $\sign\bigl(\arctan(h(t))-\arctan\phi\bigr)=\sign(\psi-\phi)$, we note that in both cases plugging the solutions inside \eqref{eq:diag-geod} yields the same expression, and we conclude that the two choices are equivalent.

We note that the explicit expression~\eqref{eq:diag-geod} for the geodesic is not symmetric with respect to the end conditions $\lambda_i(0)$ and $\lambda_i(1)$. To obtain an expression which is symmetric with respect to the end conditions we recall that $h(t) = (1 - t) \phi + t \psi$, and hence,
\begin{align*}
1 + h^2(t) &= 1 + (1 - t)^2 \phi^2 + t^2 \psi^2 + 2t(1-t) \phi \psi \\
&= (1 - t)^2(1 + \phi^2) + t^2 (1 + \psi^2) + 2t(1-t) (1 + \phi \psi) \\
& = (1 - t) [ (1 - t) (1 + \phi^2) + t (1 + \phi \psi) ] + t [ t (1 + \psi^2) + (1 - t) (1 + \phi \psi) ].
\end{align*}
Therefore,
\[
\frac{1+h(t)^2}{1+\phi^2} = (1 - t) \left[ (1 - t) + t \frac{1 + \phi \psi}{1 + \phi^2} \right] + t \frac{1 + \psi^2}{1 + \phi^2} \left[ t + (1 - t) \frac{1 + \phi \psi}{1 + \psi^2} \right].
\]
From \eqref{eq:diag-geod} it follows
\begin{equation*} 
\lambda_i^{n\beta}(1) = \lambda_i^{n\beta}(0)\frac{1+\psi^2}{1+\phi^2} \exp(b_i(\arctan \psi -\arctan \phi)),
\end{equation*}
or, equivalently,
\begin{equation*} 
\lambda_i^{n\beta}(0) \frac{1+\psi^2}{1+\phi^2} = \lambda_i^{n\beta}(1) \exp(-b_i(\arctan \psi -\arctan \phi)).
\end{equation*}
Then, \eqref{eq:diag-geod} can be written as
\begin{align*}
\lambda_i^{n\beta}(t) &= \lambda_i^{n\beta}(0)  (1 - t) \left[ (1 - t) + t \frac{1 + \phi \psi}{1 + \phi^2} \right] \exp(b_i(\arctan h(t) - \arctan \phi)) \\
& \qquad + \lambda_i^{n\beta}(1)  t \sigma \left[ t + (1 - t) \frac{1 + \phi \psi}{1 + \psi^2} \right] \exp(b_i(\arctan h(t) - \arctan \psi)).
\end{align*}
In order to complete the proof, we use the arctangent summation formula, that for $x,y$ such that $x>0$, and $x>y$, can be written as
\[
\arctan x -\arctan y = \arctan\,2(x-y,1+xy).
\]
We will consider $\phi$ and $\psi$ as in \eqref{eq:phi-psi}, for which $\phi>0$ and $\psi<0$, and for $t\in(0,1)$ we have $\phi-h(t)=t(\phi-\psi)>0$ and $h(t)-\psi=(1-t)(\phi-\psi)>0$, and thus
\begin{align*}
& \arctan\phi-\arctan h(t) = \arctantwo(\phi-h(t),1+\phi h(t)), \\
& \arctan(-\psi)-\arctan (-h(t)) = \arctantwo(-\psi+h(t),1+\psi h(t)).
\end{align*}
Finally, using the expression of $b_i$ above,
\[
\begin{split}
b_i(\arctan h(t)-\arctan\phi) & =-\sign(a)2\sqrt{n(1-n\beta)}\frac{\zeta_i}{\|\zeta\|}(\arctan h(t)-\arctan\phi)\\
& = -\frac{n\beta}{\gamma}\sign(a)\sign(\beta)\zeta_i\arctantwo(\phi-h(t),1+\phi h(t))\\
& = \frac{n\beta}{\gamma}\zeta_i\arctantwo(t\sigma\sin\gamma,1-t+t\sigma\cos\gamma),
\end{split}
\]
where we have used the fact that $\sign(a)\sign(\beta)=\sign(\psi-\phi)=-1$. Analogously, 
\[
b_i(\arctan h(t)-\arctan\psi) = -\frac{n\beta}{\gamma}\zeta_i\arctantwo((1-t)\sin\gamma,(1-t)\cos\gamma+t\sigma),
\]
and simple manipulation yields the required expression for the geodesic.
\end{proof}

\end{document}